\documentclass[a4paper,11pt]{article}
\usepackage[top=3cm,bottom=3cm,left=2.5cm,right=2.5cm]{geometry}
\usepackage[algoruled,linesnumbered]{algorithm2e}
\usepackage{algorithmic}

\usepackage{amsfonts,epsf,amsmath,amssymb,here, graphicx}
\usepackage{latexsym,multirow,rotating}
\usepackage{pgf,tikz,subfigure}
\usepackage{epstopdf}

\newtheorem{theorem}{\bf Theorem}[section]
\newtheorem{corollary}[theorem]{\bf Corollary}
\newtheorem{lemma}[theorem]{\bf Lemma}

\newtheorem{definition}[theorem]{\bf Definition}

\newcommand{\proof}{\noindent{\bf Proof.\ }}
\newcommand{\qed}{\hfill $\square$ \bigskip}
\newcommand{\s}{\color{black}}

\begin{document}

\baselineskip=0.30in
\vspace*{40mm}

\begin{center}
{\LARGE \bf Generalized cut method for computing the edge-Wiener index}
\bigskip \bigskip

{\large \bf Niko Tratnik
}
\bigskip\bigskip

\baselineskip=0.20in

Faculty of Education, University of Maribor, Slovenia \\
Faculty of Natural Sciences and Mathematics, University of Maribor, Slovenia \\
{\tt niko.tratnik@um.si, niko.tratnik@gmail.com}
\medskip

\bigskip\medskip

(Received October 28, 2019)

\end{center}

\noindent
{\bf Abstract}

\vspace{3mm}\noindent
{\s The edge-Wiener index of a connected graph $G$ is defined as the Wiener index of the line graph of $G$. In this paper it is shown that the edge-Wiener index of an edge-weighted graph can be computed in terms of the Wiener index, the edge-Wiener index, and the vertex-edge-Wiener index of weighted quotient graphs which are defined by a partition of the edge set that is coarser than $\Theta^*$-partition. Thus, already known analogous methods for computing the edge-Wiener index of benzenoid systems and phenylenes are greatly generalized. Moreover, reduction theorems are developed for the edge-Wiener index and the vertex-edge-Wiener index since they can be applied in order to compute a corresponding index of a (quotient) graph from the so-called reduced graph. Finally, the obtained results are used to find the closed formula for the edge-Wiener index of an infinite family of graphs.
}
\vspace{5mm}

\baselineskip=0.30in



\section{Introduction}
The \textit{Wiener index} of a connected graph $G$ is defined as 
$$W(G) = \sum_{\lbrace u,v \rbrace \subseteq V(G)}d_G(u,v) = \frac{1}{2}\sum_{u \in V(G)} \sum_{v \in V(G)} d_G(u,v),$$
where $d_G(u,v)$ represents the distance between vertices $u$ and $v$ of $G$ (naturally, $d_G(u,v)$ is defined as the number of edges on a shortest path from $u$ to $v$). This index was developed by Wiener in 1947 and  has many applications in chemistry and also in network theory  since it is used for analysing the structure of communication, social, and other networks \cite{redukcija}. It is also closely related to the average distance of a graph. As a consequence, the Wiener index was investigated in many papers (see \cite{knor1} for a survey).

On the other hand, the edge-Wiener index was introduced 10 years ago in papers \cite{dankelmann-2009,iranmanesh-2009,khalifeh-2009}. In~\cite{iranmanesh-2009} the authors concluded that the distance between edges $e,f$ of a graph $G$, here denoted by $d_G^0(e,f)$, should be defined as the distance between the vertices $e$ and $f$ in the line graph $L(G)$ of $G$, i.e.
$$d_G^0(e,f) = d_{L(G)}(e,f).$$
The reason for such a definition is in the fact that $(E(G),d_G^0)$ represents a metric space. Therefore, the \textit{edge-Wiener index} of a connected graph $G$ is defined as
$$W_e(G) =  \frac{1}{2}\sum_{e \in E(G)} \sum_{f \in E(G)} d_G^0(e,f).$$
We can easily see from the above definition that $W_e(G)$ is actually the Wiener index of the line graph of $G$, i.e.\ $W_e(G) = W(L(G))$. Because of this equality, many results on the edge-Wiener index were known before the index was defined (see \cite{dobrynin-1999,gutman-1996}). 

Nevertheless, if $e=xy$, $f=ab$ are two edges of a graph $G$, we can also set
$$d_G^1(e,f)= \min \lbrace d_G(x,a), d_G(x,b), d_G(y,a), d_G(y,b) \rbrace,$$
which was denoted by $\widehat{d}_G(e,f)$ in \cite{kelenc}. The above definition leads to a different version of the edge-Wiener index, denoted by $\widehat{W}_e(G)$, which is defined in the following way:
$$\widehat{W}_e(G)= \frac{1}{2}\sum_{e \in E(G)} \sum_{f \in E(G)} d_G^1(e,f).$$

\noindent
Note that in \cite{iranmanesh-2009} the numbers $W_e(G)$ and $\widehat{W}_e(G)$ were denoted by $W_{e0}(G)$ and $W_{e1}(G)$, respectively. It is easy to observe that for two edges $e,f$ of a graph $G$, $e \neq f$, it holds $d_G^0(e,f)=d_G^1(e,f) + 1$ and therefore, the following relation holds for $W_e(G)$ and $\widehat{W}_e(G)$ (see also \cite{iranmanesh-2009,khalifeh-2009}): 
\begin{equation}
\label{eq:simple-connection}
W_e(G) = \widehat{W}_e(G) + \binom{|E(G)|}{2}.
\end{equation}

For our purposes, $d_G^1$ turns out to be more  convenient than $d_G^0$. Hence, we also write $d_G$ instead of $d_G^1$, i.e.\ for any two edges $e,f \in E(G)$ we set
$$d_G(e,f)=d_G^1(e,f).$$  

For some of the recent discoveries on the edge-Wiener index the reader can refer to \cite{chen,knor-2014b,knor,soltani-2014} and also to a survey paper \cite{iranmanesh-2015}. Moreover, it is worth mentioning that the edge-Wiener index  of any connected graph $G$ is closely related to the \textit{edge-Hosoya polynomial} \cite{behm} of $G$, defined as $H_e(G,x) = \sum_{k \geq 0} d_e(G,k)\,x^k$, where $d_e(G,k)$ denotes the number of (unordered) edge pairs at distance $k$. 

The cut method is a very common and popular method for calculating molecular descriptors. Usually, it reduces the problem of calculating a topological index to the problem of calculating some indices of smaller graphs obtained by the edge cuts, see \cite{klavzar-2015} for a recent survey. This method is often applied on benzenoid systems, partial cubes, and other families of molecular graphs \cite{aroc_cle,tratnik1,trat_steiner}. In particular, some methods for computing the edge-Wiener index were developed, for instance, in \cite{arock,cre-trat1,yousefi-azari-2011}.

In \cite{kelenc} it has been shown that the edge-Wiener index of a benzenoid system can be computed by using the three Wiener indices of three weighted quotient trees obtained from elementary cuts. Later, a similar result has been established for phenylenes \cite{zigert-2018} (note that phenylenes and benzenoid systems represent important chemical graphs). In this paper, we generalize the results from \cite{kelenc,zigert-2018} and prove that the edge-Wiener index of an edge-weighted graph can be calculated in terms of the Wiener index, the edge-Wiener index, and the vertex-edge-Wiener index of weighted quotient graphs defined by a partition of the edge set that is coarser than the $\Theta^*$-partition (definitions of these concepts are included in the next section). Therefore, our method is not restricted to some specific family of graphs and neither to partial cubes, but can be applied on any graph with at least two $\Theta^*$-classes. Consequently, the mentioned result can be used to develop very fast algorithms for calculating the edge-Wiener index of important chemical graphs or networks and also to easily find formulas in the closed form for some families of graphs. Such methods were recently developed also for other distance-based topological indices: the Wiener index \cite{klavzar-2016}, the revised (edge-)Szeged index \cite{li}, the degree distance \cite{brez-trat}, the Graovac-Pisanski index \cite{tratnik_grao}. We should mention that instead of using our method for the edge-Wiener index one could use the method from \cite{klavzar-2016} on the line graph. However, it turns out that the line graphs of interesting networks often contain a small number of $\Theta^*$-classes and therefore, this method does not improve efficiency of calculating the edge-Wiener index on such graphs.

As already mentioned, with our method it is possible to calculate the edge-Wiener index by calculating the corresponding indices on quotient graphs. However, it turns out that in some cases such a quotient graph can be transformed into the  smaller graph in such a way that the indices of the first graph can be computed in terms of the corresponding indices of the reduced graph. In \cite{redukcija} it was shown that the Wiener index of a vertex-weighted graph can be computed from the Wiener index of a reduced graph. Therefore, we prove such results also for the edge-Wiener index and the vertex-edge-Wiener index.

The paper is divided into six sections. In the next section, we include some basic definitions and results that are needed later. In Section 3, it is firstly shown how the distance between two edges can be computed by using the corresponding distances in quotient graphs. In addition, we use this result to obtain the method for computing the edge-Wiener index. Moreover, the reduction theorems are proved in Section 4. Furthermore, in Section 5 the obtained results are applied to an infinite family of graphs in order to calculate the closed formula for the edge-Wiener index. Finally, some concluding remarks are mentioned in the last section.

\section{Preliminaries}

The graphs appearing in this paper are simple and finite. For a vertex $v$ of a graph $G$, we denote by $N_G(v)$, or shortly by $N(v)$, the set of vertices adjacent to $v$, i.e.\ $N_G(v) = \lbrace u \in V(G) \,|\,uv \in E(G) \rbrace$. The distance between two vertices and the two distances between two edges have already been defined in the previous section. In addition, the distance between a vertex $v \in V(G)$ and an edge $e=xy \in E(G)$ is 
\begin{equation*}
{d}_G(v,e) = \min \lbrace d_G(v,x), d_G(v, y) \rbrace\,.
\end{equation*}
The {\em vertex-edge-Wiener index} of a graph $G$ was defined in \cite{kelenc} as
$$W_{ve}(G) = \sum_{v \in V(G)} \sum_{e \in E(G)} d_G(v,e).$$
However, in~\cite{khalifeh-2009} the definition of $W_{ve}(G)$ was stated with factor $1/2$.
\smallskip

\noindent
Next, we introduce the Wiener index, both versions of the edge-Wiener index, and the vertex-edge-Wiener index of weighted graphs. Let $\mathbb{R}_0^+ = [0, \infty)$. Suppose $G$ is a graph and $w:V(G)\rightarrow {\mathbb R}_0^+$, $w_e:E(G)\rightarrow {\mathbb R}_0^+$ are some weights on the set of vertices and edges, respectively. We say that $(G,w)$ is the {\em vertex-weighted graph}, $(G,w_e)$ is the {\em edge-weighted graph}, and $(G,w,w_e)$ is the {\em vertex-edge-weighted graph}. Therefore, it is possible to define Wiener indices of such graphs as follows \cite{kelenc}:
\begin{eqnarray*}
W(G,w) & = & \frac{1}{2} \sum_{u \in V(G)} \sum_{v \in V(G)} w(u)w(v)d_G(u,v), \\
W_e(G,w_e) & = & \frac{1}{2} \sum_{e \in E(G)} \sum_{f \in E(G)} w_e(e)w_e(f)d_{G}^0(e,f),\\
\widehat{W}_e(G,w_e) & = & \frac{1}{2} \sum_{e \in E(G)} \sum_{f \in E(G)} w_e(e)w_e(f)d_G(e,f),\\
W_{ve}(G,w,w_e) & = & \sum_{v \in V(G)} \sum_{e \in E(G)} w(v)w_e(e){d}_G(v,e). 
\end{eqnarray*}

\noindent
Let $e = xy$ and $f = ab$ be two edges of a graph $G$. If $$d_G(x,a) + d_G(y,b) \neq d_G(x,b) + d_G(y,a),$$
we say that $e$ and $f$ are in relation $\Theta$ (also known as Djokovi\' c-Winkler relation) and write $e \Theta f$. Note that in some graphs this relation is not transitive (for example in odd cycles), although it is always reflexive and symmetric. As a consequence, we often consider the smallest transitive relation that contains relation $\Theta$ (i.e.\ the transitive closure of $\Theta$) and denote it by $\Theta^*$. It is known that in a \textit{partial cube}, which is defined as an isometric subgraph of some hypercube, relation $\Theta$ is always transitive, so $\Theta = \Theta^*$. Moreover, the class of partial cubes contains many interesting molecular graphs (for example benzenoid systems and phenylenes). For more information on partial cubes and relation $\Theta$ see \cite{klavzar-book}.
\smallskip

\noindent
Let $ \mathcal{E} = \lbrace E_1, \ldots, E_t \rbrace$ be the $\Theta^*$-partition of the edge set $E(G)$ and $ \mathcal{F} = \lbrace F_1, \ldots, F_r \rbrace$ an arbitrary partition of $E(G)$. If every element of $\mathcal{E}$ is a subset of some element of $\mathcal{F}$, we say that $\mathcal{F}$ is \textit{coarser} than $\mathcal{E}$. In such a case $\mathcal{F}$ will be shortly called a \textit{c-partition}. 
\smallskip

\noindent
Suppose $G$ is a graph and $E' \subseteq E(G)$ is some subset of its edges. The \textit{quotient graph} $G / E'$ is defined as the graph that has connected components of $G \setminus E'$ as vertices; two such components $X$ and $Y$ being adjacent in $G / E'$ if and only if some vertex from $X$ is
adjacent to a vertex from $Y$ in graph $G$. If $F=XY \in E(G/E')$ is an edge in graph $G/E'$, then we denote by $F$ also the set of edges of $G$ that have one end vertex in $X$ and the other end vertex in $Y$, i.e. $F=  \lbrace xy \in E(G)\,|\,x \in V(X), y \in V(Y) \rbrace $. 
\smallskip

\noindent
Let $G$ be a connected graph and let $\lbrace F_1, \ldots, F_r \rbrace$ be a c-partition of the set $E(G)$. For any $i \in \lbrace 1, \ldots, r \rbrace$, we define the function $\ell_i: V(G) \rightarrow V(G/F_i)$ as follows: for any $v \in V(G)$ let $\ell_i(v)$ be the connected component $X$ of the graph $G \setminus F_i$ such that $v \in V(X)$. The next lemma was obtained in \cite{klavzar-2016}, but the proof can be also found in \cite{tratnik_grao}.

\begin{lemma} \cite{klavzar-2016,tratnik_grao} \label{distance}
Suppose $G$ is a connected graph and $\lbrace F_1, \ldots, F_r \rbrace$ is a c-partition of the set $E(G)$. If $u,v \in V(G)$ are two vertices, then 
$$d_G(u,v) = \sum_{i=1}^r d_{G / F_i}(\ell_i(u),\ell_i(v)).$$
\end{lemma}

\section{The method for computing the edge-Wiener index}
\label{sec:main}

In this section we prove that the edge-Wiener index of a graph can be computed from the corresponding quotient graphs with specific weights on vertices and edges. Firstly, we define a function which maps any edge of $G$ either to a vertex or to an edge of a quotient graph. Note that our definition generalizes the corresponding definition from \cite{kelenc}.

\begin{definition}
Suppose $G$ is a connected graph and $\lbrace F_1,\ldots, F_r  \rbrace$ is a c-partition of the set $E(G)$. For any $i \in \lbrace 1, \ldots, r \rbrace$, we introduce the function  $\alpha_i : E(G) \rightarrow V(G / F_i) \cup E(G / F_i)$ by
\begin{equation*}
\alpha_i(xy) = \left\{ 
   \begin{array}{l l}
     \ell_i(x) \in V(G/F_i); & \quad \ell_i(x)=\ell_i(y)\,,\\
     \ell_i(x) \ell_i(y) \in E(G/F_i); & \quad \ell_i(x) \neq \ell_i(y)\,,
   \end{array} \right.
\end{equation*}
where $xy$ is an arbitrary edge of $G$.
\end{definition}

\noindent
We can now show how the distance between two edges can be computed from the distances in the quotient graphs. In the proof, we use some ideas from \cite{kelenc,li} but many new observations are needed as well.

\begin{theorem}
\label{thm:dist-between-edges}
If $G$ is a connected graph and $\lbrace F_1,\ldots, F_r  \rbrace$ a c-partition of the set $E(G)$, then for every $e,f \in E(G)$,
$${d}_G(e,f)= \sum_{i=1}^r {d}_{G/F_i}(\alpha_i(e), \alpha_i(f))\,.$$
\end{theorem}

\proof
Let $e=xy$ and $f=ab$. Moreover, we can assume (without loss of generality) that ${d}_G(e,f)={d}_G(x,a)$. By Lemma \ref{distance} it follows
$$d_G(e,f) = \sum_{i=1}^r d_{G/F_i}(\ell_i(x),\ell_i(a)).$$
To finish the proof, we show that  ${d}_{G/F_i}(\alpha_i(e), \alpha_i(f)) = d_{G/F_i}(\ell_i(x),\ell_i(a))$ for any $i \in \lbrace 1, \ldots, r \rbrace$. Therefore, choose an arbitrary $i \in \lbrace 1, \ldots, r \rbrace$ and consider the following four cases.
\begin{itemize}
\item [] {\bf Case 1.} $\ell_i(x) = \ell_i(y)$ and $\ell_i(a) = \ell_i(b)$.\\
We obtain $\alpha_i(e) = \ell_i(x)$ and $\alpha_i(f)=\ell_i(a)$. Therefore, the result follows.

\item [] {\bf Case 2.} $\ell_i(x) \neq \ell_i(y)$ and $\ell_i(a) = \ell_i(b)$.\\
Obviously, $e \in F_i$ and for any $j \in \lbrace 1,\ldots,r \rbrace$, $j \neq i$, it holds $\ell_j(x)=\ell_j(y)$. Thus, we obtain $ d_{G/F_j}(\ell_j(x),\ell_j(a)) = d_{G/F_j}(\ell_j(y),\ell_j(a))$ for all $j \neq i$. Moreover, by Lemma \ref{distance} it follows
\begin{eqnarray*}
d_G(x,a) - d_G(y,a) & = & \sum_{j=1}^r d_{G/F_j}(\ell_j(x),\ell_j(a)) - \sum_{j=1}^r d_{G/F_j}(\ell_j(y),\ell_j(a)) \\
& = & \sum_{j=1}^r \big( d_{G/F_j}(\ell_j(x),\ell_j(a)) - d_{G/F_j}(\ell_j(y),\ell_j(a)) \big) \\
& = & d_{G/F_i}(\ell_i(x),\ell_i(a)) - d_{G/F_i}(\ell_i(y),\ell_i(a)).
\end{eqnarray*}
Since $d_G(x,a) \leq d_G(y,a)$, we now deduce $d_{G/F_i}(\ell_i(x),\ell_i(a)) \leq d_{G/F_i}(\ell_i(y),\ell_i(a))$ and therefore, $$d_{G/F_i}(\alpha_i(e),\alpha_i(f))=d_{G/F_i}(\ell_i(x)\ell_i(y),\ell_i(a)) = d_{G/F_i}(\ell_i(x),\ell_i(a)).$$

\item [] {\bf Case 3.} $\ell_i(x) = \ell_i(y)$ and $\ell_i(a) \neq \ell_i(b)$.\\
This case is similar to Case 2.

\item [] {\bf Case 4.} $\ell_i(x) \neq \ell_i(y)$ and $\ell_i(a) \neq \ell_i(b)$.\\
It is clear that $e,f \in F_i$ and for any $j \in \lbrace 1,\ldots,r \rbrace$, $j \neq i$, it holds $\ell_j(x)=\ell_j(y)$ and $\ell_j(a)=\ell_j(b)$. Suppose that $d_G(e,b) = d_G(z,b)$, where $z \in \lbrace x,y \rbrace$. We can calculate
\begin{eqnarray*}
d_G(e,a) - d_G(e,b) & = & d_G(x,a) - d_G(z,b) \\
& = & \sum_{j=1}^r d_{G/F_j}(\ell_j(x),\ell_j(a)) - \sum_{j=1}^r d_{G/F_j}(\ell_j(z),\ell_j(b)) \\
& = & \sum_{j=1}^r \big( d_{G/F_j}(\ell_j(x),\ell_j(a)) - d_{G/F_j}(\ell_j(z),\ell_j(b)) \big) \\
& = & d_{G/F_i}(\ell_i(x),\ell_i(a)) - d_{G/F_i}(\ell_i(z),\ell_i(b)). 
\end{eqnarray*}
Since $d_G(e,a) = d_G(x,a) \leq d_G(e,b)$, we now get 
\begin{equation} \label{neenakost}
d_{G/F_i}(\ell_i(x),\ell_i(a)) \leq d_{G/F_i}(\ell_i(z),\ell_i(b)).
\end{equation}
Obviously, $d_G(x,a) \leq d_G(y,a)$, $d_G(z,b) \leq d_G(x,b)$, and $d_G(z,b) \leq d_G(y,b)$. By using similar reasoning as in Case 2, we can show
\begin{eqnarray*}
d_{G/F_i}(\ell_i(x),\ell_i(a)) & \leq & d_{G/F_i}(\ell_i(y),\ell_i(a)), \\
d_{G/F_i}(\ell_i(z),\ell_i(b)) & \leq & d_{G/F_i}(\ell_i(x),\ell_i(b)), \\
d_{G/F_i}(\ell_i(z),\ell_i(b)) & \leq & d_{G/F_i}(\ell_i(y),\ell_i(b)).
\end{eqnarray*}
From these inequalities and from inequality \eqref{neenakost} we conclude $$d_{G/F_i}(\alpha_i(e),\alpha_i(f))=d_{G/F_i}(\ell_i(x)\ell_i(y),\ell_i(a)\ell_i(b)) = d_{G/F_i}(\ell_i(x),\ell_i(a)).$$
\end{itemize}

In each case it holds ${d}_{G/F_i}(\alpha_i(e), \alpha_i(f)) = d_{G/F_i}(\ell_i(x),\ell_i(a))$, which is what we wanted to prove. 
\qed

Let $(G,w_e)$ be a connected edge-weighted graph and $\lbrace F_1,\ldots, F_r  \rbrace$ a c-partition of the set $E(G)$. The graphs $G/F_i$, $i \in \lbrace 1, \ldots,r \rbrace$, are extended to weighted graphs $(G/F_i, w^i)$, $(G/F_i, w_e^i)$, $(T_i, w^i, w_e^i)$ in the following way: 
\begin{itemize}
\item for each $X \in V(G/F_i)$, we define $w^i(X)$ as the sum of all the weights of edges in the connected component $X$ of $G \setminus F_i$, i.e.\ $w^i(X) = \sum_{e \in E(X)}w_e(e)$;
\item for each $F = XY \in E(G/F_i)$, we define $w_e^i(F)$ as the sum of all the weights of edges in $G$ that have one end vertex in $X$ and the other end vertex in $Y$, i.e.\ $w_e^i(F) = \sum_{e \in F}w_e(e)$.

\end{itemize}
The main theorem can now be stated. 

\begin{theorem}
\label{thm:edge-Wiener}
If $(G,w_e)$ is an edge-weighted connected graph and $\lbrace F_1,\ldots, F_r  \rbrace$ a c-partition of the set $E(G)$, then
$$\widehat{W}_e(G,w_e)= \sum_{i=1}^r \left( W(G/F_i, w^i) +  \widehat{W}_e(G/F_i, w_e^i) + W_{ve}(G/F_i, w^i, w_e^i)\right).$$
\end{theorem}

\proof
By Theorem~\ref{thm:dist-between-edges} we get 
\begin{eqnarray*}
\widehat{W}_e(G,w_e) & = & \frac{1}{2} \sum_{e \in E(G)} \sum_{f \in E(G)} w_e(e)w_e(f){d}_G(e,f) \\
 &= & \frac{1}{2} \sum_{e \in E(G)} \sum_{f \in E(G)} w_e(e)w_e(f) \Bigg( \sum_{i = 1}^r {d}_{G/F_i}(\alpha_i(e), \alpha_i(f)) \Bigg) \\
  & = & \sum_{i=1}^r \Bigg( \frac{1}{2} \sum_{e \in E(G)} \sum_{f \in E(G)} w_e(e)w_e(f) {d}_{G/F_i}(\alpha_i(e), \alpha_i(f)) \Bigg)\,.
\end{eqnarray*}

For any $i \in \lbrace 1, \ldots, r \rbrace $, we denote by $E_1^i$ and $E_2^i$ the set of edges of $G$ that are mapped by function $\alpha_i$ to a vertex or to an edge, respectively. More precisely,
$$E_1^i(G) = \lbrace e \in E(G) \,|\,\alpha_i(e) \in V(G/F_i)\rbrace, \quad E_2^i(G) = \lbrace e \in E(G) \,|\,\alpha_i(e) \in E(G/F_i)\rbrace. $$

\noindent
Obviously, it holds $E_1^i \cup E_2^i = E(G)$ and $E_1^i \cap E_2^i = \emptyset$ for any $i \in \lbrace 1, \ldots, r \rbrace $. Therefore, for two distinct edges of $G$ we have three possibilities: both edges belong to $E_1^i$, both edges belong to $E_2^i$, or one edge belongs to $E_1^i$ and the other belongs to $E_2^i$. Consequently, we arrive to the next formula:
\begin{eqnarray*}
\widehat{W}_e(G,w_e) & =  &   
  \sum_{i=1}^r \Bigg( \frac{1}{2} \sum_{e \in E_1^i(G)} \sum_{f \in E_1^i(G)} w_e(e)w_e(f) {d}_{G/F_i}(\alpha_i(e), \alpha_i(f)) \\
 &+ &  \frac{1}{2} \sum_{e \in E_2^i(G)} \sum_{f \in E_2^i(G)} w_e(e)w_e(f) {d}_{G/F_i}(\alpha_i(e), \alpha_i(f)) \\
 & + &  \sum_{e \in E_1^i(G)} \sum_{f \in E_2^i(G) } w_e(e)w_e(f) {d}_{G/F_i}(\alpha_i(e), \alpha_i(f)) \Bigg)\,. 
\end{eqnarray*}

\noindent
Let $X,Y$ be two arbitrary distinct connected components of $G \setminus F_i$. Obviously, for any $e,e' \in E(X)$ and $f,f' \in E(Y)$ it holds $d_{G/F_i}(X,Y)=d_{G/F_i}(\alpha_i(e),\alpha_i(f))=d_{G/F_i}(\alpha_i(e'),\alpha_i(f'))$. Moreover,
\begin{eqnarray*}
\sum_{e \in E(X)} \sum_{f \in E(Y)} w_e(e)w_e(f){d}_{G/F_i}(\alpha_i(e),\alpha_i(f)) & = & {d}_{G/F_i} (X,Y) \sum_{e \in E(X)}w_e(e)\sum_{f \in E(Y)}w_e(f) \\ 
& = & w^i(X)w^i(Y) {d}_{G/F_i} (X,Y). 
\end{eqnarray*}

\noindent
Let $E$, $F$ be two arbitrary distinct edges of the graph $G/F_i$. Obviously, for any $e,e' \in E$ and $f, f' \in F$ it holds $d_{G/F_i}(E,F)=d_{G/F_i}(\alpha_i(e),\alpha_i(f))=d_{G/F_i}(\alpha_i(e'),\alpha_i(f'))$. Moreover,
\begin{eqnarray*}
\sum_{e \in E} \sum_{f \in F} w_e(e)w_e(f){d}_{G/F_i}(\alpha_i(e),\alpha_i(f)) & = & {d}_{G/F_i} (E,F) \sum_{e \in E}w_e(e)\sum_{f \in F}w_e(f) \\ 
& = & w_e^i(E)w_e^i(F) {d}_{G/F_i} (E,F). 
\end{eqnarray*}

\noindent
Finally, let $X$ be a connected component of $G \setminus F_i$ and $F$ an edge of the graph $G/F_i$. Obviously, for any $e,e' \in E(X)$ and $f,f' \in F$ it holds $d_{G/F_i}(X,F)=d_{G/F_i}(\alpha_i(e),\alpha_i(f))=d_{G/F_i}(\alpha_i(e'),\alpha_i(f'))$. Moreover,
\begin{eqnarray*}
\sum_{e \in E(X)} \sum_{f \in F} w_e(e)w_e(f){d}_{G/F_i}(\alpha_i(e),\alpha_i(f)) & = & {d}_{G/F_i} (X,F) \sum_{e \in E(X)}w_e(e)\sum_{f \in F}w_e(f) \\ 
& = & w^i(X)w_e^i(F) {d}_{G/F_i} (X,F). 
\end{eqnarray*}

\noindent
From the obtained observations we finally conclude 
\begin{eqnarray*}
\widehat{W}_e(G,w_e) & = & 
\sum_{i=1}^r \Bigg(  \frac{1}{2} \sum_{X \in V(G/F_i)} \sum_{Y \in V(G/F_i)} w^i(X)w^i(Y){d}_{G/F_i}(X,Y)\\
& + & \frac{1}{2} \sum_{E \in E(G/F_i)} \sum_{F \in E(G/F_i)} w_e^i(E)w_e^i(F){d}_{G/F_i}(E, F)  \\
 & + &   \sum_{X \in V(G/F_i)} \sum_{F \in E(G/F_i)} w^i(X)w_e^i(F){d}_{G/F_i}(X, F) \Bigg) \\ 
 & = & \sum_{i=1}^r \left(  W(G/F_i, w^i) + \widehat{W}_e(G/F_i, w_e^i) + W_{ve}(G/F_i, w^i, w_e^i)\right).
\end{eqnarray*} \qed

If we set $w_e(e)=1$ for all $e \in E(G)$, the following corollary follows by equation \eqref{eq:simple-connection}.

\begin{corollary}
\label{posl_izreka}
If $G$ is a connected graph and $\lbrace F_1,\ldots, F_r  \rbrace$ a c-partition of the set $E(G)$, then
$$W_e(G)= \sum_{i=1}^r \left( W(G/F_i, w^i) +  \widehat{W}_e(G/F_i, w_e^i) + W_{ve}(G/F_i, w^i, w_e^i)\right) + {|E(G)| \choose 2},$$
where $w^i: V(G/F_i) \rightarrow \mathbb{R}_0^+$, $w_e^i: E(G/F_i) \rightarrow \mathbb{R}_0^+$ are defined as follows: $w^i(X)$ is the number of edges in the connected component $X$ and $w_e^i(XY)$ is the number of edges in $G$ that have one end vertex in $X$ and the other end vertex in $Y$.
\end{corollary}

\section{Reduction theorems}

It was shown in \cite{redukcija} that in some cases the computation of $W(G,w)$ can be done by computing the Wiener index of a smaller graph obtained by a special reduction. Such a reduction can be defined, for example, by the so-called relation $R$, which is significant in various research areas (vertices in this relation are sometimes called \textit{twin vertices}). In this section, we develop analogous results for $\widehat{W}_e(G,w_e)$ and for $W_{ve}(G,w,w_e)$, since these indices are needed to efficiently compute the edge-Wiener index in terms of Corollary \ref{posl_izreka}. An example showing how the mentioned reductions can be used is provided in the next section.
\smallskip

\noindent
Let $G$ be a graph and let $u$, $v$ be two vertices of $G$. If $N(u) = N(v)$, then $u$ and $v$ are in relation $R$. It is easy to check that $R$ is an equivalence
relation on the set $V(G)$. The equivalence class of a vertex $v$ with respect to relation $R$ is denoted by $[v]_R$. The following theorem was obtained in \cite{redukcija} 

\begin{theorem} \cite{redukcija} \label{posebna_redukcija0}
Let $(G,w)$ be a connected vertex-weighted graph, $c \in V(G)$, and $C=[c]_R = \{c_1,\ldots,c_k\}$. If $(G',w')$ is defined by $G'=G \setminus (C \setminus \lbrace c \rbrace)$, $w'(c) = \sum_{x \in C}w(x)$, and $w'(v)=w(v)$ for all $v \in V(G) \setminus C$, then $$W(G,w)= W(G',w') + \sum_{\{c_i,c_j\} \subseteq C}2w(c_i)w(c_j).$$
Moreover, if $w(c_i)=a$ for any $i \in \lbrace 1, \ldots, k \rbrace$, $a \in \mathbb{R}_0^+$, then $${W}(G,w)= {W}(G',w') + a^2k(k-1).$$
\end{theorem}

However, in \cite{redukcija} it was assumed that a weight $w$ is a function to the set $\mathbb{R}^+ = (0, \infty)$, but the same proof works also when $w: V(G) \rightarrow \mathbb{R}_0^+$.

For the rest of the section, we define the graph $G'$ as in Theorem \ref{posebna_redukcija0}. For any $c \in V (G)$, let $C = [c]_R = \lbrace c_1, \ldots, c_k \rbrace$ and $G'$ the graph defined by $G' = G \setminus (C \setminus \{c \})$. Let us denote $N(c)= N(c_i)= \lbrace n_1, \ldots, n_s \rbrace$ for any $i \in \lbrace 1, \ldots,k \rbrace$. Moreover, we denote
$$I(c_i) = \lbrace c_i n_j \in E(G)\,|\,j \in \lbrace 1, \ldots, s \rbrace\rbrace,$$
$$I(n_j) = \lbrace c_i n_j \in E(G)\,|\,i \in \lbrace 1, \ldots, k \rbrace\rbrace$$
for any $i \in \lbrace 1, \ldots,k \rbrace$ or $j \in \lbrace 1, \ldots, s \rbrace$. In other words, $I(c_i)$ is the set of edges that are incident with vertex $c_i$ and $I(n_j)$ is the set of edges that are incident to $n_j$ and a vertex from $C$. We also set $I(C) = \cup_{i \in \lbrace 1, \ldots,k \rbrace} I(c_i) = \cup_{j \in \lbrace 1, \ldots,s \rbrace} I(n_j)$, which represents the set of edges that are incident to a vertex of $C$. In addition, for any $i \in \lbrace 1,\ldots,k \rbrace$, $j \in \lbrace 1,\ldots,s \rbrace$, let $I(C)_{ij}= I(C) \setminus (I(c_i) \cup I(n_j))$.

Furthermore, for any weight $w:V(G) \rightarrow \mathbb{R}_0^+$ we define the weight $w': V(G') \rightarrow \mathbb{R}_0^+$ by $w'(c)= \sum_{i=1}^k w(c_i)$ and $w'(v)=w(v)$ for any $v \in V(G) \setminus C$. Finally, for any weight $w_e: E(G) \rightarrow \mathbb{R}_0^+$ we define the weight $w_e': E(G') \rightarrow \mathbb{R}_0^+$ in the following way: $w_e'(cn_j)= \sum_{i=1}^k w_e(c_in_j)$, $j \in \lbrace 1, \ldots, s \rbrace$, and $w_e'(e) = w_e(e)$
for any $e \in E(G) \setminus I(C)$. Now we can state our results.

\begin{theorem} \label{posebna_redukcija1}
If $(G,w_e)$ is a connected edge-weighted graph, $c \in V(G)$, $C=[c]_R = \{c_1,\ldots,c_k\}$, and $N(c)= \lbrace n_1, \ldots, n_s \rbrace$, then $$\widehat{W}_e(G,w_e)= \widehat{W}_e(G',w_e') + \frac{1}{2} \sum_{c_in_j \in I(C)} \sum_{e \in I(C)_{ij}} w_e(c_in_j)w_e(e).$$
Moreover, if $w_e(e)=a$ for all $e \in I(C)$, $a \in \mathbb{R}_0^+$, then $$\widehat{W}_e(G,w_e)= \widehat{W}_e(G',w_e') + \frac{a^2ks}{2}(k-1)(s-1).$$
\end{theorem}

\begin{proof}
It is obvious that the theorem holds if $|C|=1$. Therefore, suppose  $c_1=c$ and $k \geq 2$. We can easily check that the following properties hold true:
	
	\begin{itemize}
		\item [$(i)$] $d_G(c_in_j,e)=d_G(c_rn_j,e)$ for any $c_i, c_r \in C$, $n_j \in N(c)$, and $e \in E(G) \setminus I(C)$,
			\item [$(ii)$] $d_G(e,f)=d_{G'}(e,f)$  for any two edges $e,f \in E(G) \setminus I(C)$,
		\item [$(iii)$] $d_G(c_in_j,c_rn_t)=1$ for any $c_in_j, c_rn_t \in I(C)$, $i \neq r$, $j \neq t$.

	\end{itemize}

\noindent
Using the obtained facts one can proceed as follows:

\begin{eqnarray*}
\widehat{W}_e(G,w_e)&=& \sum_{\{e,f\} \subseteq E(G)} w_e(e)w_e(f)d_G(e,f) \\
&=& \sum_{e \in E(G) \setminus I(C)}\sum_{j=1}^s \sum_{i=1}^k w_e(e)w_e(c_in_j) d_G(c_in_j,e)  \\
&+& \sum_{\{e,f\} \subseteq E(G)\setminus I(C)}w_e(e)w_e(f)d_G(e,f)  \\
&+& \sum_{\{e,f\} \subseteq I(C)}w_e(e)w_e(f)d_G(e,f)  \\
&=& \sum_{e \in E(G) \setminus I(C)} \sum_{j=1}^s w_e(e)  d_G(cn_j,e)  \sum_{i=1}^k w_e(c_in_j)  \\
&+& \sum_{\{e,f\} \subseteq E(G)\setminus I(C)}w_e(e)w_e(f)d_G(e,f)  \\
&+& \frac{1}{2} \sum_{c_in_j \in I(C)} \sum_{e \in I(C)_{ij}} w_e(c_in_j)w_e(e). 
\end{eqnarray*}
Therefore, we finally deduce

\begin{eqnarray*}
\widehat{W}_e(G,w_e) &=& \sum_{e \in E(G') \setminus I(c)} \sum_{j=1}^s w_e'(e)w_e'(cn_j) d_{G'}(cn_j,e)  \\
&+& \sum_{\{e,f\} \subseteq E(G')\setminus  I(c)}w_e'(e)w_e'(f)d_{G'}(e,f)  \\
&+& \frac{1}{2} \sum_{c_in_j \in I(C)} \sum_{e \in I(C)_{ij}} w_e(c_in_j)w_e(e) \\
& = & \widehat{W}_e(G',w_e') + \frac{1}{2} \sum_{c_in_j \in I(C)} \sum_{e \in I(C)_{ij}} w_e(c_in_j)w_e(e).
\end{eqnarray*}

\noindent
Furthermore, when all the edges from $I(C)$ have the same weight $a$, the last term in the above formula can obviously be simplified to $\frac{a^2ks}{2}(k-1)(s-1)$. \qed
\end{proof}

\noindent
Next, we prove a similar result also for $W_{ve}(G,w,w_e)$.
\begin{theorem} \label{posebna_redukcija2}
If $(G,w,w_e)$ is a connected weighted graph, $c \in V(G)$, $C=[c]_R = \{c_1,\ldots,c_k\}$, and $N(c)= \lbrace n_1, \ldots, n_s \rbrace$, then
$$W_{ve}(G,w,w_e)= W_{ve}(G',w',w_e') +  \sum_{i=1}^k \sum_{c_r \in C \setminus \lbrace c_i \rbrace} \sum_{e \in I(c_r)} w(c_i)w_e(e).$$
Moreover, if $w(c_i)=a$ for any $i \in \lbrace 1, \ldots, k \rbrace$ and $w_e(e)=b$ for all $e \in I(C)$, $a,b \in \mathbb{R}_0^+$, then $$W_{ve}(G,w,w_e)= W_{ve}(G',w',w_e') + abk(k-1)s.$$
\end{theorem}

\begin{proof}
It is obvious that the theorem holds if $|C|=1$. Therefore, suppose  $c_1=c$ and $k \geq 2$. We can easily check that the following properties hold true:
	
	\begin{itemize}
		\item [$(i)$] $d_G(c_i,e)=d_G(c_r,e)$ for any $c_i, c_r \in C$ and $e \in E(G) \setminus I(C)$,
		\item [$(ii)$] $d_G(v,c_in_j)=d_G(v,c_rn_j)$ for any $v \in V(G)\setminus C$, $c_i, c_r \in C$, and $n_j \in N(c)$,
			\item [$(iii)$] $d_G(v,e)=d_{G'}(v,e)$ for any $v \in V(G) \setminus C$ and $e \in E(G) \setminus I(C)$,
		\item [$(iv)$] $d_G(c_i,c_rn_j)=1$ for any $c_i,c_r \in C$, $i \neq r$, and $n_j \in I(c_r)$.
\end{itemize}
	
\noindent
Using these facts one can proceed as follows:

\begin{eqnarray*}
W_{ve}(G,w,w_e)&=& \sum_{v \in V(G)} \sum_{e \in E(G)} w(v)w_e(e)d_G(v,e) \\
&=&  \sum_{e \in E(G) \setminus I(C)} \sum_{i=1}^k  w(c_i)w_e(e) d_G(c_i,e)  \\
&+& \sum_{v \in V(G) \setminus C} \sum_{j=1}^s \sum_{i=1}^k w(v)w_e(c_in_j)d_G(v,c_in_j)  \\
&+& \sum_{v \in V(G) \setminus C} \sum_{e \in E(G) \setminus I(C)} w(v)w_e(e)d_G(v,e)  \\
& + & \sum_{i=1}^k  \sum_{e \in I(C)} w(c_i)w_e(e)d_G(c_i,e)\\ 
&=&  \sum_{e \in E(G) \setminus I(C)}w_e(e)d_G(c,e) \sum_{i=1}^k  w(c_i)  \\
&+& \sum_{v \in V(G) \setminus C} \sum_{j=1}^s w(v)d_G(v,cn_j) \sum_{i=1}^k w_e(c_in_j)  \\
&+& \sum_{v \in V(G) \setminus C} \sum_{e \in E(G) \setminus I(C)} w(v)w_e(e)d_G(v,e)  \\
& + & \sum_{i=1}^k \sum_{c_r \in C \setminus \lbrace c_i \rbrace} \sum_{n_j \in N(c_r)} w(c_i)w_e(c_rn_j).
\end{eqnarray*}
Hence, we finally deduce
\begin{eqnarray*}
{W}_{ve}(G,w,w_e) &=& \sum_{e \in E(G') \setminus I(c)}  w'(c)w_e'(e) d_{G'}(c,e)  \\
&+& \sum_{v \in V(G') \setminus \lbrace c \rbrace} \sum_{j=1}^s w'(v)w_e'(cn_j)d_{G'}(v,cn_j) \\
&+& \sum_{v \in V(G') \setminus \lbrace c \rbrace} \sum_{e \in E(G') \setminus I(c)} w'(v)w_e'(e)d_{G'}(v,e)  \\
&+& \sum_{i=1}^k \sum_{c_r \in C \setminus \lbrace c_i \rbrace} \sum_{e \in I(c_r)} w(c_i)w_e(e) \\
& = & {W}(G',w',w_e') + \sum_{i=1}^k \sum_{c_r \in C \setminus \lbrace c_i \rbrace} \sum_{e \in I(c_r)} w(c_i)w_e(e).
\end{eqnarray*} 

\noindent
In addition, if $w(c_i)=a$ for any $i \in \lbrace 1, \ldots, k \rbrace$ and $w_e(e)=b$ for all $e \in I(C)$, the last term obviously equals $abk(k-1)s$. \qed

\end{proof}

\section{An example}

Here we apply the obtained results on an infinite family of graphs to show how they can be used to easily calculate the edge-Wiener index. For $m,n \geq 1$, let $G_{m,n}$ be the graph shown in Figure \ref{graphGnm} which has $m$ horizontal and $n$ vertical layers of hexagons. We obtain $|V(G_{m,n})| = (m+1)(2n+1)$ and $|E(G_{m,n})|=3mn+m+2n$. Note that this family of graphs includes many important molecular graphs. In particular, when $m=1$, graph $G_{1,n}$ is the linear benzenoid chain with $n$ hexagons (also called polyacene) \cite{gucy-89}. Moreover, if $n=1$, then the family $G_{m,1}$ represents fused cyclohexane rings: for an example, $G_{2,1}$ is the molecular graph of bicyclo[3.3.1]nonane \cite{abe}. 

\begin{figure}[h!] 
\begin{center}
\includegraphics[scale=0.7,trim=0cm 0.5cm 0cm 0cm]{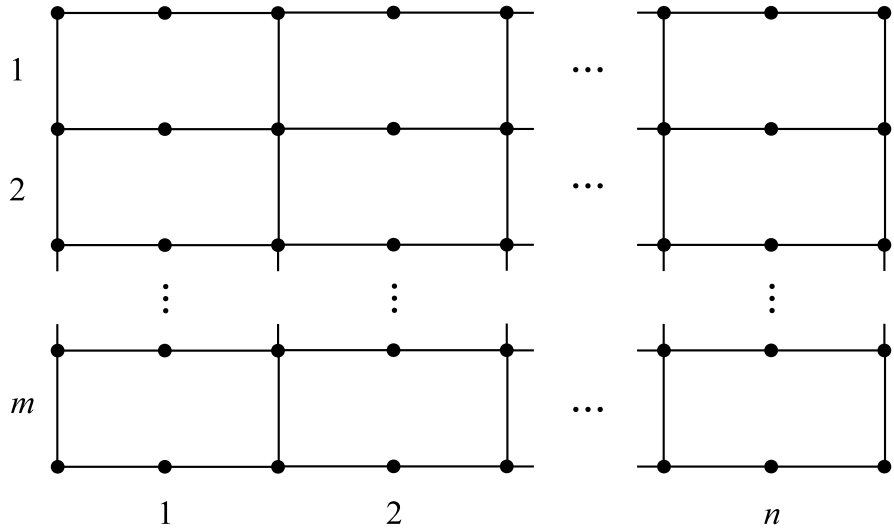}
\end{center}
\caption{\label{graphGnm} Graph $G_{m,n}$.}
\end{figure}

Firstly, we have to determine the $\Theta^*$-classes of $G_{m,n}$. It is easy to see that any two edges on a shortest path can not be in relation $\Theta$. Moreover, any two diametrically opposite edges in an isometric even cycle of a graph are in relation $\Theta$. By using these facts, it is not difficult to check that graph $G_{m,n}$, where $m \geq 2$, has $\Theta^*$-classes $D_1, \ldots, D_m$ and $E_1, \ldots, E_{n}$ as shown in Figure \ref{theta_classes}. Moreover, for $m=1$, the set $E_i$, $i \in \lbrace 1, \ldots, n \rbrace$, in Figure \ref{theta_classes} is composed of two $\Theta$-classes. Furthermore, it is easy to see that for $m \geq 2$ graph $G_{m,n}$ is not a partial cube since relation $\Theta$ is not transitive.

\begin{figure}[h!] 
\begin{center}
\includegraphics[scale=0.7,trim=0cm 0.5cm 0cm 0cm]{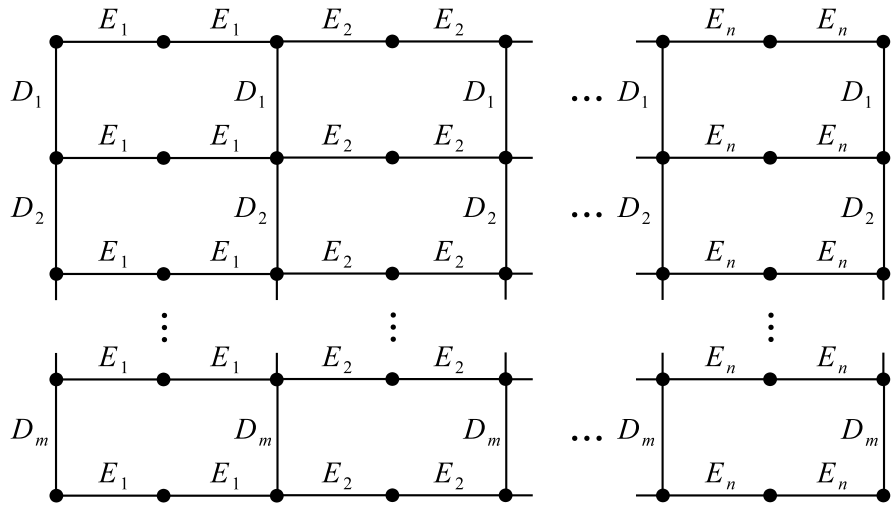}
\end{center}
\caption{\label{theta_classes} $\Theta^*$-classes of $G_{m,n}$, where $m \geq 2$.}
\end{figure}

Next, we define $F_1 = \cup_{i=1}^m D_i$ and $F_2 = \cup_{i=1}^n E_i$. Obviously, $\lbrace F_1,F_2 \rbrace$ is a c-partition of the set $E(G_{m,n})$. The quotient graph $G_{m,n}/F_1$ is isomorphic to the path on $m+1$ vertices and the weights $w^1,w_e^1$ are calculated as in Corollary \ref{posl_izreka}, see Figure \ref{quotient1}.

\begin{figure}[h!] 
\begin{center}
\includegraphics[scale=0.7,trim=0cm 0.5cm 0cm 0cm]{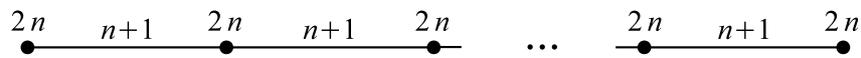}
\end{center}
\caption{\label{quotient1} Weighted quotient graph $(G_{m,n}/F_1, w^1, w_e^1)$, $G_{m,n}/F_1 \cong P_{m+1}$.}
\end{figure}

\noindent
Therefore, we compute:

\begin{eqnarray*}
W(G_{m,n}/F_1,w^1) & = & 4n^2 \sum_{i=1}^{m} \sum_{j=i+1}^{m+1}(j-i) = \frac{2n^2}{3}\left ( m^3 + 3m^2 + 2m \right), \\
\widehat{W}_e(G_{m,n}/F_1,w_e^1) & = & (n+1)^2 \sum_{i=1}^{m-1} \sum_{j=i+1}^{m}(j-i-1) = \frac{(n+1)^2}{6}\left ( m^3 - 3m^2 + 2m \right), \\
{W}_{ve}(G_{m,n}/F_1,w^1, w_e^1) & = & 2n(n+1) \left( \sum_{i=2}^{m+1} \sum_{j=1}^{i-1}(i-j-1) + \sum_{i=1}^{m}\sum_{j=i}^{m} (j-i) \right) \\
& = & \frac{2n(n+1)}{3}\left ( m^3 -m \right).
\end{eqnarray*}

On the other hand, the quotient graph $G_{m,n}/F_2$ is depicted in Figure \ref{quotient2}. Weights $w^2,w_e^2$ are again calculated in terms of Corollary \ref{posl_izreka}. The sets of vertices $C^1, \ldots, C^n$ are shown in the same figure and each of them contains $m+1$ vertices.

\begin{figure}[h!] 
\begin{center}
\includegraphics[scale=0.7,trim=0cm 0.5cm 0cm 0cm]{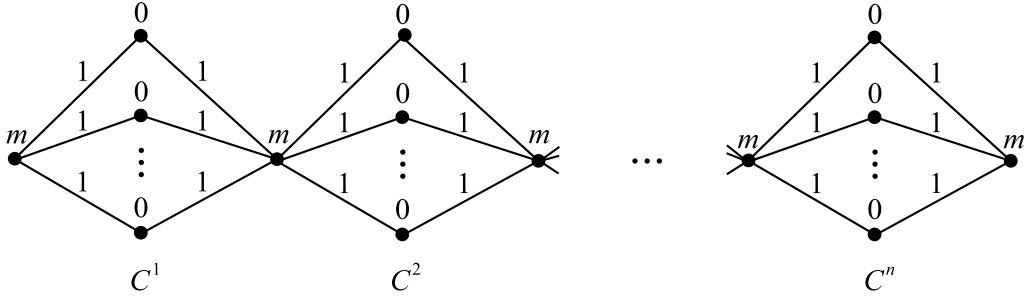}
\end{center}
\caption{\label{quotient2} Weighted quotient graph $(G_{m,n}/F_2, w^2, w_e^2)$ with sets of vertices $C^1, \ldots, C^n$.}
\end{figure}

Note that the sets $C^i$, $i \in \lbrace 1, \ldots, n \rbrace$, represent some of the equivalence classes of $R$. Therefore, we perform the reduction from Section 4 exactly $n$ times on $G_{m,n}/F_2$ and obtain the weighted graph which will be denoted as $(H, w^3,w_e^3)$. Obviously, $H$ is isomorphic to the path on $2n+1$  vertices, see Figure \ref{reduction}.

\begin{figure}[h!] 
\begin{center}
\includegraphics[scale=0.7,trim=0cm 0.5cm 0cm 0cm]{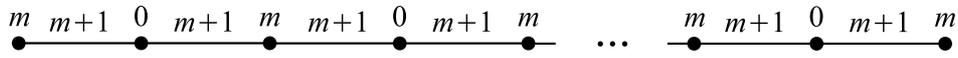}
\end{center}
\caption{\label{reduction} Weighted graph $(H, w^3, w_e^3)$, $H \cong P_{2n+1}$.}
\end{figure}

\noindent
Hence, we can calculate:
\begin{eqnarray*}
W(H,w^3) & = & m^2 \sum_{i=0}^{n-1} \sum_{j=i+1}^{n} ((2j+1) - (2i+1)) = \frac{m^2}{3} \left( n^3 + 3n^2 + 2n \right), \\
\widehat{W}_e(H,w_e^3) & = & (m+1)^2 \sum_{i=1}^{2n-1} \sum_{j=i+1}^{2n}(j-i-1) = \frac{2(m+1)^2}{3}\left ( 2n^3 - 3n^2 + n \right), \\
{W}_{ve}(H,w^3, w_e^3) & = & m(m+1) \left( \sum_{i=1}^{n} \sum_{j=1}^{2i}(2i+1 - j - 1) + \sum_{i=0}^{n-1}\sum_{j=2i+1}^{2n} (j-2i-1) \right) \\
& = & \frac{m(m+1)}{3}\left ( 4n^3 + 3n^2 - n \right).
\end{eqnarray*}

\noindent
By Theorems \ref{posebna_redukcija0}, \ref{posebna_redukcija1}, and \ref{posebna_redukcija2} we obtain
\begin{eqnarray*}
W(G_{m,n}/F_2,w^2) & = & W(H,w^3) + 0 \cdot n, \\
\widehat{W}_e(G_{m,n}/F_2,w_e^2) & = & \widehat{W}_e(H,w_e^3) + m(m+1)n, \\
{W}_{ve}(G_{m,n}/F_2,w^2, w_e^2) & = & W_{ve}(H, w^3, w_e^3) + 0 \cdot n.
\end{eqnarray*}

\noindent
Finally, by Corollary \ref{posl_izreka} it follows
\begin{eqnarray*}
W_e(G_{m,n}) & = &  \frac{1}{6} \big( 9m^3n^2 + 18m^2n^3 + 6m^3n + 36m^2n^2 + 24mn^3 + m^3   \\
&+& 24m^2n + 24mn^2 + 8n^3 + 15mn - m -2n \big).
\end{eqnarray*}
\section{Concluding remarks}
In the paper, we have developed a method for computing the edge-Wiener index, which is a well investigated distance-based topological descriptor. Our method is useful for any connected graph with at least two $\Theta^*$-classes and reduces the problem of computing the edge-Wiener index to the problem of computing three Wiener indices of weighted quotient graphs. Therefore, it generalizes analogous results known for benzenoid systems and phenylenes. Since in some cases such quotient graphs can be further transformed into smaller (also called reduced) graphs, we have proved that the edge-Wiener index and the vertex-edge-Wiener index of a weighted graph can be computed from the reduced graphs. These results were applied in Section 5, where the closed formula for the edge-Wiener index of an infinite family of graphs was deduced. 

It is worth mentioning that the obtained results can be used to  efficiently calculate the edge-Wiener index for many important families of molecular graphs. Moreover, in Section 3 we have proved a formula for computing the distance between two edges in a graph from the distances in the quotient graphs, which can be applied in order to develop new methods for computing some other distance-based topological indices.
   
\section*{Acknowledgement} 
\noindent
The author was financially supported by the Slovenian Research Agency (research core funding No. P1-0297 and J1-9109).

\end{document}